\documentclass[10pt,a4paper]{article}
\usepackage[utf8]{inputenc}
\usepackage{amsmath}
\usepackage{amsfonts}
\usepackage{amssymb}
\usepackage{amsthm} 

\usepackage{graphicx}
\usepackage{caption}
\usepackage{subcaption}
\usepackage{color}
\usepackage[final]{showkeys}

\newtheoremstyle{examplestyleone}  
  {3mm}       
  {10mm}       
  {\normalfont}   
  {}        
  {\bfseries}  
  {\quad}   
  {1mm}       
  {}           

\theoremstyle{examplestyleone}

\newtheoremstyle{examplestyletwo}  
  {5mm}       
  {3mm}       
  {\itshape}   
  {}        
  {\bfseries}  
  {}   
  {3mm}       
  {}           

\theoremstyle{examplestyletwo}
\newtheorem{theorem}{Theorem}[section]
\newtheorem{lemma}[theorem]{Lemma}

\newtheorem{remark}[theorem]{Remark}

\everymath{\displaystyle} 

\DeclareMathOperator{\var}{var}

\numberwithin{equation}{section}

\title{Multilevel Monte Carlo analysis for optimal control of elliptic PDEs with random coefficients} 
\author{Ahmad Ahmad Ali\footnotemark[1], \ Elisabeth Ullmann\footnotemark[2] \ and Michael Hinze\footnotemark[1]}

\begin{document}
\maketitle
\renewcommand{\thefootnote}{\fnsymbol{footnote}}
\footnotetext[2]{Fachbereich Mathematik, Universit\"{a}t Hamburg, Bundesstr. 55, 20146 Hamburg, Germany (\{ahmad.ali,michael.hinze\}@uni-hamburg.de).} 
\footnotetext[1]{Zentrum Mathematik M2, TU M\"{u}nchen, Boltzmannstr. 3, 85748 Garching, Germany 
(elisabeth.ullmann@ma.tum.de).}
\renewcommand{\thefootnote}{\arabic{footnote}}

\begin{abstract}
This work is motivated by the need to study the impact of data uncertainties and material imperfections on the solution to optimal control problems constrained by partial differential equations.  
We consider a pathwise optimal control problem constrained by a diffusion equation with random coefficient together with box constraints for the control.
For each realization of the diffusion coefficient we solve an optimal control problem using the variational discretization  [M. Hinze, Comput. Optim. Appl., 30 (2005), pp. 45-61].
Our framework allows for lognormal coefficients whose realizations are not uniformly bounded away from zero and infinity.
We establish finite element error bounds for the pathwise optimal controls.
This analysis is nontrivial due to the limited spatial regularity and the lack of uniform ellipticity and boundedness of the diffusion operator.
We apply the error bounds to prove convergence of a multilevel Monte Carlo estimator for the expected value of the pathwise optimal controls.
In addition we analyze the computational complexity of the multilevel estimator.
We perform numerical experiments in 2D space to confirm the convergence result and the complexity bound.
\end{abstract}

 \medskip


\section{Introduction}\label{sec:intro}

Optimization problems constrained by partial differential equations (PDEs) arise from many applications in engineering and science, e.g. the optimal control of fluid flows, or shape optimization.
Since the boundary conditions or material parameters are often not precisely known we consider optimal control problems (OCPs) constrained by PDEs with random or parametrized coefficients.
By this we mean that the coefficients, source terms and/or boundary values associated  with the PDE are modeled as random fields or parameter-dependent functions where we postulate a certain probability distribution for the parameters.
The motivation for this setup is uncertainty quantification (UQ) in complex PDE-based simulations where it is crucial to account for imprecise or missing information in the input data.

Elliptic PDEs with random coefficients are well studied to date and there is a large body of literature on efficient approximation techniques and solvers, see e.g. \cite{BNT:2010,GunzburgerActa}, \cite[Chapter 9]{LPSBook} and the references therein.
However, PDE-constrained optimization problems with random coefficients have been studied only very recently.
We consider a distributed control problem with tracking-type cost functional constrained by an elliptic PDE with a random coefficient together with box constraints for the control. 
A similar problem setup with elliptic and parabolic PDE constraint has been considered in \cite{Borzi:2010} and \cite{BorziWinckel:2009}, respectively.

For each fixed realization of the random coefficient in the PDE constraint we solve an OCP. 
We then estimate the statistics, e.g. the expected value or variance of the ensemble of pathwise optimal controls.
This allows us to study the sensitivity of the controls with respect to random fluctuations in the PDE constraint and provides practical information on the design of control devices subject to uncertain inputs.
It is clear, however, that the expected value of the pathwise optimal controls does not solve an OCP in general and is not necessarily a robust control.
We give further motivation and background information on various OCPs with random coefficients in \S\ref{sec:background}.

In this paper we describe and analyze a multilevel Monte Carlo (MLMC) estimator for the expected value of the pathwise optimal controls.
MLMC is an efficient and versatile alternative to standard Monte Carlo (MC) estimation.
It has been successfully used to estimate the statistics of output functionals generated by elliptic PDEs with random coefficients; see e.g. the pioneering works \cite{BSZ:2011,CGST:2011}.
In further developments, MLMC has been used for e.g. mixed finite element (FE) approximations of random diffusion problems \cite{GSU:2015}, random obstacle problems \cite{BierigChernov:2014,KSW:2014}, Markov chains \cite{DKST:2015}, or for the approximation of the distribution function of random variables \cite{GNR:2015}.  

The basic idea of MLMC is \textit{variance reduction}.
The MLMC estimator employs a hierarchy of approximations to the desired output quantity associated with certain levels.
From one level to the next the approximation quality of the output increases but so does the cost to compute it.
By a clever combination of additive corrections with decreasing variance the total cost of the MLMC estimator can be reduced dramatically compared to standard MC.
By construction, the variance of the corrections is large on coarse levels and small on fine levels.
Hence the MLMC estimator requires a large number of inexpensive output samples with a coarse resolution and only a modest number of highly resolved, expensive output samples. 
In our setting the levels are defined by FE discretizations of the pathwise optimal control problems on nested triangulations of the computational domain.

The major contribution of this work is a careful error analysis of the MLMC approximation to the expected value of the pathwise optimal controls.
This requires bounds on the FE error and the sampling error.
We extend the classical FE analysis of PDE-based OCPs with deterministic inputs (see e.g. \cite{pinnau2008optimization}) to a setting with random PDE coefficients. 
Along the way we make novel contributions.
For example, we prove that the pathwise optimal controls are realizations of a well-defined random field.
This requires nontrivial arguments from set-valued analysis.
We mention that other statistical information, e.g. the variance of the optimal controls, could also be estimated using the multilevel Monte Carlo framework (cf. \cite{BierigChernov:2014}).

To analyze the total approximation error we begin by carrying out a FE error analysis for a fixed realization of the random PDE coefficient.
Then we extend the error estimates to the entire sample space using H\"{o}lder's inequality.
This procedure is well established by now and follows e.g. \cite{charrier2013finite,teckentrup2013further}.
However, it is nontrivial for two reasons.
First, it requires the careful tracking of all constants in the standard FE error estimates.
These ``generic'' constants are usually not precisely specified in deterministic analyses.
In our setting they cannot be ignored since they depend on the realization of the random coefficient in the PDE constraint.
Second, we assume that the realizations of the random coefficient are not uniformly bounded over the sample space and are only H\"{o}lder continuous with exponent $0<t\leq 1$.
Hence the underlying elliptic PDE in the pathwise OCP is not uniformly elliptic with respect to the random inputs.
This scenario is typical for lognormal coefficients $a$ where $\log(a)$ is a mean-zero Gaussian random field with a certain non-smooth covariance function.

The FE error analysis of PDE-based OCPs relies heavily on the FE error associated with the state approximation.
Of course, the limited smoothness of the PDE coefficient results in a limited regularity of the state; standard textbook $H^2$-regularity is only achieved for $t=1$.
This reduces the convergence rate of the $L^2$-error of the pathwise optimal control.
This situation is not ideal since the variance reduction and hence the efficiency of the MLMC estimator is determined by the FE error convergence rate of the output quantity of interest; the faster the FE error decays the larger is the variance reduction.
For this reason we employ the variational discretization \cite{hinze2005variational} for the control together with piecewise linear continuous finite elements for the state.
Then it can be proved that the $L^2$-error of the (pathwise) optimal control converges with rate $2s$ for any $0<s<t$.
This is in fact the optimal rate in our this setting.
Alternatively, for the classical discretizations with piecewise constant or piecewise linear continuous controls the $L^2$-error of the pathwise optimal control converges only with the rates $s$ and $3s/2$, respectively, in the best case (cf. \cite{Roesch:2006} and \cite[\S~3.2.6]{pinnau2008optimization}).

The remainder of this paper is organized as follows. 
In \S\ref{sec:background} we review OCPs with random coefficients and comment on the state of the art.
In \S\ref{sec:problem} we formulate the pathwise OCP together with the state equation and first order optimality conditions.
The regularity of the state and the optimal control is also investigated.
In \S\ref{sec:FE} we focus on the discretization of the pathwise OCP. 
We discuss the variational discretization and prove an priori estimate for the $L^2$-error of the pathwise optimal controls. 
In \S\ref{sec:MC} we describe MC and multilevel MC estimators for the expected value of the pathwise optimal controls and analyze the computational costs of these estimators.
Finally, in \S\ref{sec:example} we discuss a 2D test problem.
We confirm numerically the $L^2$-error bound for the expectation of the controls and the complexity bound for the MLMC estimator.

\section{Background}\label{sec:background}

Various formulations for PDE-constrained OCPs with random coefficients have appeared in the literature to date. 
These can be distinguished by assumptions on the optimal control (deterministic or stochastic) and the form of the cost functional whose minimum is sought.
In the following we give an informal overview to put our work in context.
In addition we comment on solvers for these problems.

Consider a cost functional $J=J(u,y(u),a)$ where $u$ denotes the control, $y$ denotes the state and $a$ is some parameter associated with the PDE constraint. 
Note that we could write $J$ in reduced form as a function of the control only. 
In our setting, $a$ is a random function with realizations denoted by $a_\omega$.
We distinguish the following problem formulations:
\begin{itemize}
\item[(a)] 
Mean-based control, see \cite{Borzi:2010,Borzietal:2010,BorziWinckel:2009}:
Replace $a$ by its expected value $\mathbb{E}[a]$. Minimize $J(u,y(u),\mathbb{E}[a])$ by a deterministic optimal control. 
\item[(b)] 
Individual or ``pathwise'' control, see \cite{Borzi:2010,BorziWinckel:2009,Negrietal:2015,Negrietal:2013}: 
Fix $a_\omega$, minimize $J(u,y(u),a_\omega)$ and obtain a realization $u^\ast_\omega$ of the stochastic optimal control $u^\ast$.
In a postprocessing step, compute the statistics of $u^\ast$, e.g. $\mathbb{E}[u^\ast]$.
\item[(c)]
Averaged control, see \cite{LazarZuaZua:2014,ZuaZua:2014}: 
Control the averaged (expected) state by minimizing $J(u,\mathbb{E}[y(u)],a)$ using a deterministic optimal control.
\item[(d)]
Robust deterministic control, see \cite{Borzietal:2010,BorziWinckel:2011,GarreisUlbrich:2016,GunzburgerMing:2011,Gunzburgeretal:2011,HouLeeManouzi:2011,Kouri:2014,Kourietal:2013,Kourietal:2014,LeeLee:2013,RosseelWells:2012}: 
Minimize the expected cost $\mathbb{E}[J(u,y(u),a)]$ by a deterministic optimal control. 
\item[(e)]
Robust stochastic control, see \cite{BennerOnwutaStoll:2015a,BennerOnwutaStoll:2015b, ChenQuarteroniRozza:2015,ChenQuarteroni:2014,ChenQuarteroniRozza:2013,KunothSchwab:2013,Tiesleretal:2012}: Minimize the expected cost $\mathbb{E}[J(u,y(u),a)]$ by a stochastic optimal control. 
\end{itemize}
The mean-based control problem (a) does not account for the uncertainties in the PDE and it is not clear if the resulting deterministic optimal control is robust with respect to the random fluctuations.
The pathwise control problem (b) is highly modular and can be combined easily with sampling methods, e.g. Monte Carlo or sparse grid quadrature. 
However, the expected value $\mathbb{E}[u^\ast]$ does not solve an OCP and is in general not a robust control. 
The average control problem (c) introduced by Zuazua \cite{ZuaZua:2014} seeks to minimize the distance of the expected state to a certain desired state.
This is an interesting alternative to the robust control problem in (d) where the expected distance of the (random) state to a desired state is minimized.
Since the cost functional in (c) uses a weaker error measure than the cost functional in (d) the average optimal control does not solve the robust control problem in general.
\textit{Stochastic} optimal controls in (e) are of limited practical use since controllers typically require a \textit{deterministic} signal.
This can, of course, be perturbed by a \textit{known} mean-zero stochastic fluctuation which models the uncertainty in the controller response.
For these reasons, deterministic, robust controls in (d) are perhaps most useful in practice and have attracted considerable attention compared to the other formulations.
However, the robust OCP in (d) or the average OCP in (c) involve an infinite number of PDE constraints which are coupled by a single cost functional.
The approximate solution of such problems is extremely challenging and requires much more computational resources than e.g. a deterministic OCP with a single deterministic PDE constraint.
For this reason it is worthwhile to explore alternative problem formulations.

In this paper we consider a pathwise control problem of the form (b).
(The precise problem formulation is given in \S\ref{sec:ocp}.)
We are interested in the statistical properties of the pathwise optimal controls, e.g. the expected value $\mathbb{E}[u^\ast]$ or variance $\var[u^\ast]$.
Observe that $\mathbb{E}[u^\ast]$ can be used as initial guess for the robust control problem in (d) if the variance $\var[u^\ast]=\mathbb{E}[u^\ast-\mathbb{E}[u^\ast]]^2$ is small.
This is justified by the Taylor expansion
\[
\mathbb{E}[\widehat{J}(u^\ast)]
=
\widehat{J}(\mathbb{E}[u^\ast])+ \frac12 \frac{d^2 \widehat{J}}{d u^2}(\mathbb{E}[u^\ast]) \var[u^\ast] \ + \ \text{higher order moments},
\]
where we have used the reduced cost functional $\widehat{J}=\widehat{J}(u)$ and the assumption that $\widehat{J}$ is smooth.
However, we re-iterate that $\mathbb{E}[u^\ast]$ is in general not the solution of an optimization problem.

The control problems (a)--(e) have been tackled by a variety of solver methodologies.
We mention stochastic Galerkin approaches in \cite{HouLeeManouzi:2011,KunothSchwab:2013,LeeLee:2013,RosseelWells:2012}, stochastic collocation in \cite{Borzi:2010,BorziWinckel:2009,ChenQuarteroniRozza:2013,Kouri:2014,Kourietal:2013,Kourietal:2014,RosseelWells:2012,Tiesleretal:2012}, low-rank, tensor-based methods in \cite{BennerOnwutaStoll:2015a,BennerOnwutaStoll:2015b,GarreisUlbrich:2016}, and reduced basis/POD methods in \cite{BorziWinckel:2011,ChenQuarteroniRozza:2015,ChenQuarteroni:2014,GunzburgerMing:2011,Negrietal:2015,Negrietal:2013}.
In this paper, we employ multilevel Monte Carlo to estimate the expected value of the pathwise controls, see \S\ref{sec:MLMC}.

\section{The problem setting}\label{sec:problem}

\subsection{Notation}
In the sequel, $(\Omega,\mathcal{A},\mathbb{P})$ denotes a probability space, where $\Omega$ is a sample space, $\mathcal{A} \subset 2^\Omega$ is a $\sigma$-algebra and $\mathbb{P}:\mathcal{A} \to [0, 1]$ is a probability measure. Given a Banach space $(X,\|\cdot\|_X)$, the  space $L^p(\Omega,X)$ is the set of strongly measurable functions $v :\Omega \to X$ such that $\|v\|_{L^p(\Omega,X)} < \infty$, where 
\[\|v\|_{L^p(\Omega,X)}:= \left\{
  \begin{array}{lr}
   \bigg(\int_\Omega \|v(\omega)\|^p_X d\mathbb{P}(\omega) \bigg)^{1/p} & \quad \mbox{ for } 1\leq  p < \infty,\\
    \mbox{esssup}_{\omega \in \Omega}\|v(\omega)\|_X  & \quad \mbox{ for } p=\infty.
  \end{array}
\right.
\]
For $L^p(\Omega,\mathbb{R})$ we write $L^p(\Omega)$. For a bounded  Lipschitz domain  $D \subset \mathbb{R}^d$, the spaces $C^0(\bar D)$ and $C^1(\bar D)$ are the usual spaces of uniformly continuous functions and continuously differentiable functions, respectively, with their standard norms. The space $C^t(\bar D)$ with $0<t<1$ denotes the space of H\"{o}lder continuous functions with the norm 
\begin{align*}
\|v\|_{C^t(\bar D)}:= \sup_{x \in \bar D} |v(x)| + \sup_{x,y \in \bar D,x \neq y} \frac{|v(x)-v(y)|}{|x-y|^t}.
\end{align*} 
For $k \in \mathbb{N}$, the space $H^k(D)$ is the classical Sobolev space, on which we define the norm and semi-norm, respectively, as 
\begin{align*}
\|v\|_{H^k(D)}:=\Bigg(\sum_{|\alpha|\leq k} \int_D |D^\alpha v|^2 \, dx \Bigg)^{1/2} \quad \mbox{and} \quad |v|_{H^k(D)}:=\Bigg(\sum_{|\alpha|= k} \int_D |D^\alpha v|^2 \, dx \Bigg)^{1/2}.
\end{align*}
Recall that, for bounded domains $D \subset \mathbb{R}^d$, the norm $\|\cdot\|_{H^k(D)}$ and semi-norm $|\cdot|_{H^k(D)}$ are equivalent on the subspace $H_0^k(D)$ of $H^k(D)$. For $r:=k+s$ with  $0<s<1$ and $k \in \mathbb{N}$, we denote by $H^r(D)$ the space of all functions $ v \in H^k(D)$ such that $\|v\|_{H^r(D)} < \infty$, where the norm $\|\cdot\|_{H^r(D)}$ is defined by 
\begin{align*}
\|v\|_{H^r(D)}:= \Bigg( \|v\|^2_{H^k(D)} + \sum_{|\alpha|=k} \int_D \int_D \frac{| D^\alpha v(x)-D^\alpha v(y)|^2}{|x-y|^{d+2s}} \, dx \, dy \Bigg)^{1/2}.
\end{align*}
Finally, for any two positive quantities $a$ and $b$, we write $a \lesssim b$ to indicate that $\tfrac{a}{b}$ is uniformly bounded by a positive constant independent of the realization $\omega\in \Omega$ or the discretization parameter $h$.

\subsection{The state equation}
Let $D \subset \mathbb{R}^d, d=1,2,3$ denote a  bounded convex domain with Lipschitz boundary $\partial D$. 
For simplicity we assume that $D$ is polyhedral.  
We consider the following linear elliptic PDE with random coefficients:
\begin{equation}
\label{Eqn:1}
 \begin{aligned}
-\nabla \cdot ( a(\omega,x) \nabla y(\omega,x) )  &=u(x)  \quad \mbox{ in }  D, \\
y(\omega, x) &=0  \quad \mbox{ on } \partial D,
\end{aligned}
\end{equation}
for $\mathbb{P}$-a.s. $\omega \in \Omega$,
 where $(\Omega,\mathcal{A},\mathbb{P})$ is a probability space.
The $\sigma$-algebra $\mathcal{A}$ associated with $\Omega$ is generated by the random variables $\{a(\cdot,x)\colon x\in D\}$. 
The differential operators $\nabla \cdot$ and $\nabla$ are with respect to $x \in D$. Let us formally define for all $\omega \in \Omega$,
\begin{equation*}
a_{\min}(\omega):=\min_{x \in \bar D} a(\omega, x)  \qquad \mbox{ and } \qquad a_{\max}(\omega):=\max_{x \in \bar D} a(\omega, x).  
\end{equation*}
We make the following assumptions on the data.
\begin{description}
\item[A1.]
$a_{\min} \geq 0$ almost surely and $a^{-1}_{\min} \in L^p(\Omega)$, for all $p \in [1, \infty)$,
\item[A2.]
$a \in L^p(\Omega,C^t(\overline D))$ for some $0< t \leq 1$ and for all $p \in [1, \infty)$,
\item[A3.]
$u \in L^2(D)$.
\end{description}
Notice that Assumption~A2 implies that the quantities $a_{\min}$ and $a_{\max}$ are well defined and that $a_{\max} \in  L^p(\Omega)$ for all $p \in [1, \infty)$ since $a_{\max}(\omega)=\|a\|_{C^0(\bar D)}$. Moreover, together with Assumption~A1, it follows that $a_{\min}(\omega)>0$ and $a_{\max}(\omega)< \infty$ for almost all $\omega \in \Omega$.

From now on, for simplicity of notation, the subscript $\omega$ denotes the dependence on $\omega$. 
The variational formulation of \eqref{Eqn:1}, parametrized by $\omega \in \Omega$, is 
\begin{equation}
\label{Eqn:2}
\int_D a_\omega \nabla  y_\omega   \cdot  \nabla v  \, dx  =  \int_D  u v \,dx \quad \forall \, v \in H^1_0(D).
\end{equation}
We say that for any $\omega \in \Omega$, $y_\omega$ is a weak solution of \eqref{Eqn:1} if and only if $y_\omega \in H^1_0(D)$ and satisfies \eqref{Eqn:2}.
For completeness we restate here a special case of \cite[Theorem 2.1]{teckentrup2013further} on the regularity of the solution to \eqref{Eqn:2}.

\begin{theorem}
\label{Thm:1}
Let Assumptions A1-A3 hold for some $0<t \leq 1$. Then, for $\mathbb{P}$-a.s. $\omega \in \Omega$, there exists a unique weak solution $y_\omega \in H^1_0(D)$ to \eqref{Eqn:1}. 
It holds
\begin{equation}
\label{Eqn:12}
|y_{\omega}|_{H^1(D)} \lesssim \frac{\|u\|_{L^2(D)}}{a_{\min}(\omega)},
\end{equation}
and
\begin{equation}
\label{Eqn:11}
\|y_\omega\|_{H^{1+s}(D)} \lesssim C_{\ref{Thm:1}}(\omega) \|u\|_{L^2(D)},
\end{equation}
for all $0<s<t$ except $s=\tfrac{1}{2}$, where
\begin{align*}
C_{\ref{Thm:1}}(\omega):=\frac{a_{\max}(\omega) \|a_\omega\|^2_{C^t(\bar D)} }{a_{\min}(\omega)^4}.
\end{align*}
Moreover, $y \in L^p(\Omega,H^{1+s}(D))$, for all $p \in [1,\infty)$. If $t=1$, then $y \in L^p(\Omega,H^{2}(D))$ and the bound \eqref{Eqn:11} holds with $s=1$.
\end{theorem}
\begin{proof}
Let's define for $\mathbb{P}$-a.s. $\omega \in \Omega$ the bilinear form $b_\omega : H^1_0(D) \times H^1_0(D) \to \mathbb{R}$ and the linear functional $l : H^1_0(D) \to \mathbb{R}$ by
\begin{align}
\label{Eqn:13} 
b_\omega (y,v) := \int_D a_\omega \nabla  y   \cdot  \nabla v  \, dx, \quad \mbox{ and } \quad l (v) := \int_D u v  \, dx.
\end{align}
It is clear that Assumption A3 implies $l \in H^{-1}(D)$ (the dual space of $H^1_0(D)$). Moreover, from Assumptions A1-A2 we have
\begin{align}
\label{Eqn:14}
|b_\omega (y,v)| & \leq a_{\max} (\omega) |y|_{H^1(D)} |v|_{H^1(D)} \quad \forall \, y , v \in H_0^1(D), \\
|b_\omega (y,y)| & \geq a_{\min} (\omega) |y|^2_{H^1(D)} \quad \forall \, y  \in H_0^1(D). \nonumber
\end{align}
Hence, according to the Lax-Milgram lemma, there exists a unique $y_\omega \in H_0^1(D)$ such that
\begin{align*}
b_\omega (y_\omega,v)=l (v) \quad \forall \, v \in H_0^1(D), \quad \mbox{ and } \quad |y_{\omega}|_{H^1(D)} \lesssim \frac{\|u\|_{L^2(D)}}{a_{\min}(\omega)}.
\end{align*}

For the $H^{1+s}(D)$ regularity of $y_\omega$ and the estimate \eqref{Eqn:11}, we refer the reader to \cite{teckentrup2013further}. From \eqref{Eqn:11}, Assumptions A1-A2 and the H\"older inequality, it follows that $y \in L^p(\Omega,H^{1+s}(D))$. This completes the proof.
\end{proof}
In the light of Theorem~\ref{Thm:1}, we introduce the following weak solution operator of \eqref{Eqn:1} for $\mathbb{P}$-a.s. $\omega \in \Omega$
\begin{equation}
\label{Eqn:4}
S_{\omega}:L^2(D) \to H^1_0(D)
\end{equation}
such that $y_\omega :=S_{\omega}u$  is the weak solution of \eqref{Eqn:1} for a given right hand side $u \in L^2(D)$ and a realization $a_\omega$. Obviously, the operator $S_\omega$ is bounded and linear.

\subsection{The optimal control problem}\label{sec:ocp}
For $\mathbb{P}$-a.s. $\omega \in \Omega$, we consider the following optimal control problem parametrized by $\omega \in \Omega$
\[
(\mathrm{P}_\omega) \quad
\begin{array}{l}
 \min_{u \in U_{ad}} J_\omega(u):=\frac{1}{2} \Vert S_{\omega} u-z \Vert_{L^2(D)}^2  + \dfrac{\alpha}{2} \Vert u \Vert_{L^2(D)} ^2 \\
\end{array}
\]
where $S_{\omega}:L^2(D) \to H^1_0(D)$ is the weak solution operator of the elliptic PDE \eqref{Eqn:1} as introduced in \eqref{Eqn:4}, $U_{ad} \subset L^2(D)$ is the closed convex set of admissible controls defined by
\begin{equation}
\label{Uad}
U_{ad}:=\{ u \in L^2(D) : u_a \leq u(x) \leq u_b \mbox{ a.e. in } D\}
\end{equation}
with $- \infty \leq u_a \leq u_b \leq \infty$ are given. Finally,  the desired state $z \in L^2(D)$ is a given deterministic function and $\alpha >0$ is a given real number.
To begin we establish the existence and uniqueness of the solution to $(\mathrm{P}_\omega)$.

\begin{theorem}
\label{Thm:10}
Suppose that $U_{ad}$ is non-empty. Then, for  $\mathbb{P}$-a.s. $\omega \in \Omega$, there exists a unique global solution $u_\omega \in U_{ad}$ for the Problem~$(\mathrm{P}_\omega)$. 
\end{theorem}
\begin{proof}
For a fixed $\omega \in \Omega$, the Problem~$(\mathrm{P}_\omega)$ is a deterministic infinite dimensional optimization problem. Since the set $U_{ad}$ is closed and convex and the cost functional $J_\omega$ is strictly convex, due to the linearity of $S_\omega$ and $\alpha>0$,
 it is enough to consider a minimizing sequence and argue in a classical way to verify the existence  of a global solution for $(\mathrm{P}_\omega)$. The uniqueness follows from the strict convexity of $J_\omega$. For the details we refer the reader to [\cite{pinnau2008optimization}, Chapter~1].
\end{proof}
 
The subscript in $u_\omega$ is only to indicate that $u_\omega$ is the solution of $(\mathrm{P}_\omega)$ for a given $\omega \in \Omega$. The next result gives the first order optimality conditions of $(\mathrm{P}_\omega)$. For the proof we refer the reader to [\cite{pinnau2008optimization}, Chapter~1].
\begin{theorem}
\label{Thm:11}
A feasible control $u_\omega \in U_{ad}$ is a solution of $(\mathrm{P}_\omega)$ if and only if there exist a state $y_\omega \in H_0^1(D)$ and an adjoint state $p_\omega \in H_0^1(D)$ such that there holds
\begin{align*}
&\int_D a_\omega \nabla  y_\omega   \cdot  \nabla v  \, dx  =  \int_D  u_\omega v \,dx \quad \forall \, v \in H^1_0(D),   \\
&\int_D a_\omega \nabla  p_\omega   \cdot  \nabla v  \, dx  =  \int_D (y_\omega -z) v \, dx   \quad \forall \, v \in H^1_0(D), \\
&\int_D (p_\omega + \alpha  u_\omega)(u-u_\omega) \, dx \geq  0 \quad \forall \, u \in U_{ad}. 
\end{align*}
\end{theorem}

Now we define a map $u^\ast: \Omega \times D \to \mathbb{R}$ where $u^\ast(\omega,\cdot)$ is the solution of $(\mathrm{P}_\omega)$ for the given $\omega \in \Omega$. 
Precisely, let
\begin{equation}
\label{Eqn:6}
u^\ast: \Omega \times D \to \mathbb{R} \, \mbox{ such that } \, u^\ast(\omega,\cdot):= \min_{u \in U_{ad}} J_\omega(u) \quad \mbox{ for } \mathbb{P}\mbox{-a.s. } \omega \in \Omega.
\end{equation}
Next we prove that $u^\ast$ is a random field and establish some properties of it.
\begin{theorem}
\label{Thm:extra}
For each $\omega \in \Omega$ the map $\omega \mapsto u^\ast(\omega,\cdot)$  is measurable.
\end{theorem}
\begin{proof}
Subsequently, we write $J(\omega,u)$ instead of $J_\omega(u)$ for any $(\omega,u) \in \Omega\times L^2(D)$ for convenience.

Recall that for a fixed control $u \in L^2(D)$ the map $\Omega \ni \omega \mapsto S_\omega u \in L^2(D)$ is measurable.
This implies, by \cite[Proposition~1.2]{da2014stochastic}, that  $\Omega \ni \omega \mapsto \|S_\omega u -z\|_{L^2(D)} \in \mathbb{R}$ is measurable as well. 
From this, we can easily see that the map $J:\Omega\times L^2(D) \to \mathbb{R}$ defined in Problem~$(\mathrm{P}_\omega)$ is Carath\'{e}odory, i.e., for every $u \in L^2(D)$, $J(\cdot,u)$ is measurable and for every $\omega \in \Omega$, $J(\omega,\cdot)$ is continuous.
Since $J$ is  Carath\'{e}odory, we deduce from \cite[Theorem~8.2.11]{aubin2009set} that the set valued map $R$ defined by
\begin{align*}
\Omega \ni \omega \mapsto R(\omega):=\{u \in U_{ad}: J(\omega,u)=\min_{v \in U_{ad}}J(\omega,v)\},
\end{align*}
is measurable and there exists a measurable selection of $R$ by \cite[Theorem~8.1.3]{aubin2009set}. 
However, since Theorem~\ref{Thm:10} guarantees that for every $\omega \in \Omega$, the set $R(\omega)$ has a unique element which we denote by $u^\ast(\omega,\cdot)$, we conclude that the map $\Omega \ni \omega \mapsto u^\ast(\omega,\cdot) \in L^2(D)$ is the measurable selection of $R$.
This is the desired conclusion.
\end{proof}

\begin{theorem}
\label{Thm:15} 
Let Assumptions A1-A3 hold for some $0<t \leq 1$ and let $u^\ast$ be the random field defined in \eqref{Eqn:6}.
Then for any $u \in U_{ad}$ there holds
\begin{align} 
\label{Eqn:32} 
\|u^\ast(\omega,\cdot)\|_{L^2(D)}\lesssim C_{\ref{Thm:15}}(\omega):= \sqrt{ \left(\frac{2}{\alpha a_{\min}(\omega)^2}+1 \right)  \|u\|^2_{L^2(D)} + \frac{2}{\alpha} \|z\|^2_{L^2(D)} }.
\end{align}
Moreover,  $ u^\ast \in L^p(\Omega, L^2(D))$ for all $p \in [1,\infty)$.
\end{theorem}
\begin{proof}
We begin by establishing the bound in \eqref{Eqn:32}.
From the optimality of $u^\ast(\omega,\cdot)$ together with the estimate \eqref{Eqn:12} it follows that for any $u \in U_{ad}$ there holds
\begin{align*}
\frac{\alpha}{2} \|u^\ast(\omega,\cdot)\|^2_{L^2(D)} &\leq J(\omega,u^\ast(\omega,\cdot)) \leq J(\omega,u)=\frac{1}{2} \Vert S_{\omega} u-z \Vert_{L^2(D)}^2  + \frac{\alpha}{2} \Vert u \Vert_{L^2(D)} ^2 \\
& \lesssim \left(\frac{1}{a_{\min}(\omega)^2}+\frac{\alpha}{2} \right)  \|u\|^2_{L^2(D)} + \|z\|^2_{L^2(D)},
\end{align*}
from which we obtain the desired result. 
Finally, from the  estimate \eqref{Eqn:32} together with  Assumption~A1 and the H\"{o}lder inequality it follows that $u^\ast \in L^p(\Omega, L^2(D))$ for all $p \in [1,\infty)$.
This completes the proof.
\end{proof}
\begin{remark}\label{rem:more}
Under the assumptions of Theorem~\ref{Thm:15} and assuming that either the bounds $u_a$ and $u_b$ are finite, or $U_{ad}$ is bounded, or $0 \in U_{ad}$, it is possible to prove that $u^\ast \in L^\infty(\Omega, L^2(D))$.
\end{remark}
\begin{remark}
The estimate in \eqref{Eqn:32} holds for any $u \in U_{ad}$.
In the proof of Theorem~\ref{Thm:3} we will see that it is desirable to choose $u \in U_{ad}$ such that the constant $C_{\ref{Thm:15}}(\omega)$ in \eqref{Eqn:32} is small. 
This can be achieved by using the projection of $0$ onto $U_{ad}$, that is $u:=\min\{\max\{0,u_a\},u_b\}$ in \eqref{Eqn:32}.
\end{remark}

\section{Finite element discretization}\label{sec:FE}
\label{sec:1}

Let $\mathcal{T}_h$ be a triangulation of $D$ with maximum mesh size $h:=\max_{T \in \mathcal{T}_h}\mbox{diam}(T) $ such that 
\[
\bar D= \bigcup_{T \in \mathcal{T}_h} \bar T. 
\]
In addition, we assume that the triangulation is quasi-uniform in the sense that there exists a constant $\kappa >0$ (independent of $h$) such that each $T \in \mathcal{T}_h$ is contained in a ball of radius $\kappa^{-1}h$ and contains a ball of radius $\kappa h$. Finally, we define the space of linear finite elements
\[
X_{h}:= \{ v_h \in C^0(\bar D) : v_h \mbox{ is a linear polynomial on each }  T \in \mathcal{T}_h \mbox{ and } {v_h}_{ | \partial D}=0   \}.
\]

\subsection{The discrete state equation}
For $\mathbb{P}$-a.s. $\omega \in \Omega$, the finite element discretization of \eqref{Eqn:2} is defined as follows: Find $y_{\omega,h} \in  X_{h}$ such that
\begin{equation}
\label{Eqn:3}
 \int_D a_\omega \nabla y_{\omega,h}  \cdot  \nabla v_h \, dx = \int_D  u v_h \,dx \quad \forall v_h \in X_{h}.
\end{equation}

\begin{theorem}
\label{Thm:7}
Let Assumptions A1-A3 hold for some $0<t \leq 1$. Then, for $\mathbb{P}$-a.s. $\omega \in \Omega$, there exists a unique solution $y_{\omega,h} \in  X_{h}$ to \eqref{Eqn:3}. 
Moreover,
\begin{equation*}
|y_{\omega,h}|_{H^{1}(D)} \lesssim \frac{\|u\|_{L^2(D)}}{a_{\min}(\omega)}.
\end{equation*}
\end{theorem}
\begin{proof}
The result follows from applying the Lax-Milgram Lemma as in the proof of Theorem~\ref{Thm:1}.
\end{proof}

Thanks to Theorem~\ref{Thm:7}, we  introduce the following solution operator of \eqref{Eqn:3} for $\mathbb{P}$-a.s. $\omega \in \Omega$ 
\begin{equation}
\label{Eqn:18}
S_{\omega,h}:L^2(D) \to X_{h}
\end{equation}
such that $y_{\omega,h} :=S_{\omega,h}u$  is the solution of \eqref{Eqn:3} for a given $u \in L^2(D)$ and a realization $a_\omega$. Notice that the operator $S_{\omega,h}$ is bounded and linear.
The following result provides the error estimate in $H^1(D)$ associated with approximating $S_{\omega}u$ by $S_{\omega,h}u$. 
\begin{theorem}
\label{Thm:8}
Let Assumptions A1-A3 hold for some $0<t \leq 1$. Then 
\begin{equation} 
\label{Eqn:15}
|S_{\omega, h}u-S_{\omega}u |_{H^1(D)} \lesssim C_{\ref{Thm:8}}(\omega) \|u\|_{L^2(D)} h^{s},
\end{equation}
for $\mathbb{P}$-a.s. $\omega \in \Omega$  and for all $0< s <t$ except $s=\tfrac{1}{2}$, where  
\[
C_{\ref{Thm:8}}(\omega):= \Big(\dfrac{a^3_{\max}(\omega)}{a^{9}_{\min}(\omega)} \Big)^{\frac{1}{2}} \|a_\omega\|^2_{C^t(\bar D)}.
\] 
If $t=1$, the above estimate holds with $s=1$. 
\end{theorem}

\begin{proof}
The statement follows by combining Theorem 2.2 and Theorem 2.1 in \cite{teckentrup2013further}.
\end{proof}

In the next theorem we investigate the error in the difference $S_{\omega}u-S_{\omega,h}u$ in $L^2(D)$, but before that we need the following lemma which can be found for instance in \cite[Chapter 8]{hackbusch1992elliptic}.

\begin{lemma}
\label{Thm:9}
Let $v \in H^1_0(D) \cap H^{1+s}(D)$ for some $0<s \leq 1$. Then
\begin{align*}
\inf_{v_h \in X_h} |v -v_h |_{H^1(D)} \lesssim  \|v \|_{H^{1+s}(D)} h^s, 
\end{align*}
where the hidden constant is independent of $v$ and $h$.
\end{lemma}
 
\begin{theorem}
\label{Thm:2}
Let Assumptions A1-A3 hold for some $0<t \leq 1$. Then 
\begin{equation}
\label{Eqn:16}
\|S_{\omega, h}u-S_{\omega}u \|_{L^2(D)} \lesssim C_{\ref{Thm:2}}(\omega) \|u\|_{L^2(D)} h^{2 s},
\end{equation}
for $\mathbb{P}$-a.s. $\omega \in \Omega$  and for all $0< s <t$ except $s=\tfrac{1}{2}$, where  
\[
C_{\ref{Thm:2}}(\omega):= \Big(\dfrac{a^7_{\max}(\omega)}{a^{17}_{\min}(\omega)} \Big)^{\frac{1}{2}} \|a_\omega\|^4_{C^t(\bar D)}.
\] 
Moreover,
\begin{equation}
\label{Eqn:17}
\|S_h u-S u \|_{L^p(\Omega, L^2(D))} \leq c h^{2 s},  \quad \mbox{ for all } p \in [1,\infty),
\end{equation}
with $c>0$ independent of $\omega$ and $h$. If $t=1$, the above two estimates hold with $s=1$. 
\end{theorem}
\begin{proof}
The key idea is a duality argument.
For $\mathbb{P}$-a.s. $\omega \in \Omega$, let $e_\omega:=S_{\omega, h}u-S_{\omega}u$ and let $\tilde y_\omega \in H^1_0(D)$ be the solution of the problem
\begin{align*}
 b_\omega(\tilde y_\omega,v)=(e_\omega,v)_{L^2(D)}:=\int_D e_\omega v \, dx \quad \forall \, v \in H^1_0(D),
\end{align*}
where $b_\omega$ is the bilinear form defined in \eqref{Eqn:13}. Observe that $\tilde y_\omega \in H^{1+s}(D)$ according to Theorem~\ref{Thm:1}.
Moreover, from the Galerkin orthogonality there holds
\begin{align*}
 b_\omega(e_\omega,v_h)=0 \quad \forall \, v_h \in X_h.
\end{align*}
Hence, we have
\begin{align*}
\|e_\omega \|^2_{L^2(D)} &=  (e_\omega, e_\omega)_{L^2(D)} =  b_\omega(\tilde y_\omega,e_\omega) =b_\omega(\tilde y_\omega -v_h,e_\omega)   \\
& \leq a_{\max}(\omega) |\tilde y_\omega -v_h |_{H^1(D)}  |e_\omega|_{H^1(D)}  \quad (\mbox{using } \eqref{Eqn:14})\\
& \leq a_{\max}(\omega) \|\tilde y_\omega \|_{H^{1+s}(D)}  |e_\omega|_{H^1(D)} h^s \quad (\mbox{using Lemma~\ref{Thm:9}})\\
& \lesssim a_{\max}(\omega)  C_{\ref{Thm:1}}(\omega)  \|e_\omega\|_{L^2(D)}  |e_\omega|_{H^1(D)} h^s \quad (\mbox{using } \eqref{Eqn:11})\\
& \lesssim a_{\max}(\omega)  C_{\ref{Thm:1}}(\omega)  \|e_\omega\|_{L^2(D)} C_{\ref{Thm:8}}(\omega) \|u\|_{L^2(D)} h^{2s} \quad (\mbox{using } \eqref{Eqn:15}).
\end{align*}
Dividing both sides of the previous inequality by $\|e_\omega \|_{L^2(D)}$ gives the estimate \eqref{Eqn:16} from which we get  \eqref{Eqn:17} after applying the H\"{o}lder inequality together with  the Assumptions A1-A2. This completes the proof.
\end{proof}

\begin{remark}
The order of convergence $O(h^{2s})$ in the estimate \eqref{Eqn:16} is obtained while assuming that the integrals in \eqref{Eqn:3} are computed exactly. In general, those integrals can't be computed exactly, instead, they are approximated by quadrature which introduces another sort of error that one should consider. However, it is still possible to achieve the order  $O(h^{2s})$ in \eqref{Eqn:16} even with quadrature provided that $a(\omega,\cdot)$  belongs to at least $C^{2s}(\bar D)$ as it was explained in \cite[\S~3.3]{charrier2013finite}. It is important to mentioned this at this stage because all the upcoming error estimates related to the optimal control problem are heavily depending on \eqref{Eqn:16}.
\end{remark}

\subsection{The discrete optimal control problem}
In this section we discretize the problem $(\mathrm{P}_\omega)$ via the variational discretization approach developed in \cite{hinze2005variational}.  
For $\mathbb{P}$-a.s. $\omega \in \Omega$, the variational discretization  of $(\mathrm{P}_\omega)$  reads
\[
(\mathrm{P}_{\omega,h}) \quad
\begin{array}{l}
 \min_{u \in U_{ad}} J_{\omega,h}(u):=\frac{1}{2} \Vert S_{\omega, h}u-z \Vert_{L^2(D)}^2  + \dfrac{\alpha}{2} \Vert u \Vert_{L^2(D)} ^2, \\
\end{array}
\]
where $S_{\omega, h}:L^2(D) \to X_{h}$ is the solution operator of \eqref{Eqn:3} as introduced in \eqref{Eqn:18}.
The key idea of the variational discretization is to discretize only the state equation by finite elements (usually piecewise linear continuous FEs) while the control is still sought in $U_{ad}$. 
Hence problem $(\mathrm{P}_{\omega,h})$ is again an optimization problem in infinite dimensions.
Thus all techniques we used previously to study problem  $(\mathrm{P}_{\omega})$ can also be used for $(\mathrm{P}_{\omega,h})$.
A detailed study of the variational discretization together with its numerical implementation and comparisons to classical discretizations can be found in \cite{hinze2005variational} and \cite[Chapter 3]{pinnau2008optimization}.

Analogously to Theorem~\ref{Thm:10}, one can show that for $\mathbb{P}$-a.s. $\omega \in \Omega$ the problem $(\mathrm{P}_{\omega,h})$ admits a unique global solution which we denote by $u_{\omega, h}$. The next theorem gives the first order optimality conditions of $(\mathrm{P}_{\omega,h})$. For the proof we refer the reader to \cite[Chapter~3]{pinnau2008optimization}.

\begin{theorem}
\label{Thm:12}
A feasible control $u_{\omega,h} \in U_{ad}$ is a solution of $(\mathrm{P}_{\omega,h})$ if and only if there exist a state $y_{\omega,h} \in X_{h}$ and an adjoint state $p_{\omega,h} \in X_{h}$ such that there holds
\begin{align*}
&\int_D a_\omega \nabla  y_{\omega,h}  \cdot  \nabla v_h  \, dx  =  \int_D  u_{\omega,h} v_h \,dx \quad \forall \, v_h \in X_{h},   \\
&\int_D a_\omega \nabla  p_{\omega,h}   \cdot  \nabla v_h  \, dx  =  \int_D (y_{\omega,h} -z) v_h \, dx   \quad \forall \, v_h \in X_{h}, \\
&\int_D (p_{\omega,h} + \alpha  u_{\omega,h})(u-u_{\omega,h}) \, dx \geq  0 \quad \forall \, u \in U_{ad}. 
\end{align*}
\end{theorem}

The next result provides a key ingredient to establish the error estimate in approximating  the solution of  $(\mathrm{P}_{\omega})$  by the one of $(\mathrm{P}_{\omega,h})$ for a given $\omega \in \Omega$.
\begin{theorem}
\label{Thm:4}
For $\mathbb{P}$-a.s. $\omega \in \Omega$, let $(u_{\omega},y_{\omega},p_{\omega}) \in L^2(D)\times H_0^1(D) \times H_0^1(D)$ satisfy the optimality conditions of $(\mathrm{P}_{\omega})$ and let  $(u_{\omega,h},y_{\omega,h},p_{\omega,h}) \in L^2(D)\times X_{h} \times X_{h}$ satisfy the optimality conditions of $(\mathrm{P}_{\omega,h})$. Then
\begin{align*}
\frac{\alpha}{2} \|u_{\omega}-u_{\omega,h}\|^2_{L^2(D)} + \frac{1}{2} \|y_{\omega}-y_{\omega,h}\|^2_{L^2(D)} & \leq \frac{1}{2} \|S_{\omega}u_\omega-S_{\omega, h}u_\omega\|^2_{L^2(D)} \\
 & \quad + \frac{1}{2\alpha} \|S_{\omega}(y_\omega-z)-S_{\omega,h}(y_\omega-z)\|^2_{L^2(D)}.
\end{align*} 
\end{theorem}
\begin{proof}
For a given $\omega \in \Omega$, the systems in Theorem~\ref{Thm:11} and Theorem~\ref{Thm:12} are deterministic. Hence, it is sufficient to apply \cite[Theorem~3.4]{pinnau2008optimization}. 
\end{proof}

Analogously to the random field $u^\ast$ defined in \eqref{Eqn:6}, we introduce the discrete random field $u_h^\ast:\Omega\times D \to \mathbb{R}$ whose realization $u_h^\ast(\omega,\cdot)$ is the solution of $(\mathrm{P}_{\omega,h})$ for the given $\omega \in \Omega$ and mesh size $h$. 
Precisely, let
\begin{equation}
\label{Eqn:7}
u_h^\ast: \Omega \times D \to \mathbb{R} \, \mbox{ such that } \, u_h^\ast(\omega,\cdot):= \min_{u \in U_{ad}} J_{\omega,h}(u) \quad \mbox{ for } \mathbb{P}\mbox{-a.s. } \omega \in \Omega.
\end{equation}
That $u^\ast_h$ is indeed a random field can be proved analogously to the proof of Theorem~\ref{Thm:extra}. Next, we establish some properties of $u^\ast_h$.
\begin{theorem}
\label{Thm:16} 
Let Assumptions A1-A3 hold for some $0<t \leq 1$ and let $u^\ast_h$ be the random field defined in \eqref{Eqn:7}.
Then for any $u \in U_{ad}$ there holds
\begin{align} 
\label{Eqn:33} 
\|u_h^\ast(\omega,\cdot)\|_{L^2(D)}\lesssim C_{\ref{Thm:16}}(\omega):= \sqrt{ \left(\frac{2}{\alpha a_{\min}(\omega)^2}+1 \right)  \|u\|^2_{L^2(D)} + \frac{2}{\alpha} \|z\|^2_{L^2(D)} }.
\end{align}
Moreover,  $ u_h^\ast \in L^p(\Omega, L^2(D))$ for all $p \in [1,\infty)$.
\end{theorem}
\begin{proof}
The proof is very similar to that of Theorem~\ref{Thm:15} with few obvious modifications. 
\end{proof}

\begin{remark}
$u_h^\ast \in L^\infty(\Omega, L^2(D))$ can be proved analogously to Remark~\ref{rem:more}.
\end{remark}

We can now show that the random field $u_h^\ast$ converges to $u^\ast$ in $L^p(\Omega, L^2(D))$ as the discretization parameter $h$ tends to zero and we derive the corresponding error estimate.
\begin{theorem}
\label{Thm:3}
Let Assumptions A1-A3 hold for some $0<t \leq 1$ and let $u^\ast , u^\ast_h \in L^p(\Omega, L^2(D))$ be the random fields defined in \eqref{Eqn:6} and \eqref{Eqn:7}, respectively. Then
\begin{equation}
\label{Eqn:29}
\|u^\ast(\omega,\cdot)- u^\ast_h(\omega,\cdot)\|_{L^2(D)} \lesssim C_{\ref{Thm:3}}(\omega) h^{2 s},
\end{equation}
\label{Eqn:19}
for almost all $\omega \in \Omega$ and for all $0< s <t$ except $s=\tfrac{1}{2}$, where
\begin{equation*}
C_{\ref{Thm:3}}(\omega):=\frac{1}{\alpha} \sqrt{\alpha \|u^\ast(\omega,\cdot)\|^2_{L^2(D)} +  \max(1,a^{-1}_{\min}(\omega))^2 \big( \|u^\ast(\omega,\cdot)\|_{L^2(D)}+\|z\|_{L^2(D)} \big)^2 }\times C_{\ref{Thm:2}}(\omega).
\end{equation*}
Moreover,
\begin{equation}
\label{Eqn:20}
\|u^\ast- u^\ast_h\|_{L^p(\Omega,L^2(D))} \leq C(\alpha,z,{U_{ad}}) h^{2 s}, \quad \mbox{ for all } p \in [1,\infty),
\end{equation}
with $C(\alpha,z,{U_{ad}})>0$ independent of $\omega$ and $h$ and depending only on the deterministic data $\alpha,z$ and $U_{ad}$. If $t=1$, the above two estimates hold with $s=1$. 
\end{theorem}

\begin{proof}
We start the proof by recalling  that for a given $\omega\in \Omega$ the realizations $u^\ast(\omega,\cdot), u^\ast_h(\omega,\cdot)$ are, by definition, the solutions of $(\mathrm{P}_{\omega})$ and $(\mathrm{P}_{\omega,h})$, respectively. Hence, utilizing   Theorem~\ref{Thm:4} we obtain 
\begin{align}
\label{Eqn:8}
\frac{\alpha}{2} \|u^\ast(\omega,\cdot)-u^\ast_h(\omega,\cdot)\|^2_{L^2(D)}  \leq \frac{1}{2} \|S_{\omega}u_\omega-S_{\omega, h}u_\omega\|^2_{L^2(D)} \nonumber\\
\quad + \frac{1}{2\alpha} \|S_{\omega}(y_\omega-z)-S_{\omega,h}(y_\omega-z)\|^2_{L^2(D)},
\end{align} 
where we define here and subsequently $u_\omega:=u^\ast(\omega,\cdot)$ and $y_\omega:=S_\omega u_\omega$.

We start by estimating the first term in  \eqref{Eqn:8}. To achieve this,  we see that from Theorem~\ref{Thm:2} we have
\begin{align*}
\|S_{\omega}u_\omega-S_{\omega, h}u_\omega\|_{L^2(D)} \lesssim C_{\ref{Thm:2}}(\omega) \|u_\omega\|_{L^2(D)} h^{2 s}. 
\end{align*}
The second term can be estimated as follows:
\begin{align*}
\|S_{\omega}(y_\omega-z)-S_{\omega,h}(y_\omega-z)\|_{L^2(D)} &\lesssim C_{\ref{Thm:2}}(\omega) \|y_\omega-z\|_{L^2(D)} h^{2 s} \quad \mbox{(from Theorem~\ref{Thm:2})}\\
&\lesssim C_{\ref{Thm:2}}(\omega) \big( \|y_\omega\|_{L^2(D)}+\|z\|_{L^2(D)} \big) h^{2 s} \\
&\lesssim C_{\ref{Thm:2}}(\omega) \big(a^{-1}_{\min}(\omega) \|u_\omega\|_{L^2(D)}+\|z\|_{L^2(D)} \big) h^{2 s} \quad \mbox{(using \eqref{Eqn:12})}\\
&\lesssim C_{\ref{Thm:2}}(\omega) \max(1,a^{-1}_{\min}(\omega)) \big( \|u_\omega\|_{L^2(D)}+\|z\|_{L^2(D)} \big) h^{2 s}. 
\end{align*}
Inserting the above estimates into \eqref{Eqn:8} gives the estimate \eqref{Eqn:29} from which one obtains \eqref{Eqn:20} after using the H\"{o}lder inequality together with Assumptions A1-A2 as well as recalling  \eqref{Eqn:32}. This completes the proof.  
\end{proof}

\section{Monte Carlo FE Methods}\label{sec:MC}
\label{sec:MC FE methods}
In this section we study the approximation of the expected value $\mathbb{E}[u^\ast]$ of the random field $u^\ast$ defined by \eqref{Eqn:6} via multilevel Monte Carlo methods. We start first by reviewing the classical Monte Carlo method, but before that we make the following assumptions and notation. 

Let $\{\mathcal{T}_{h_l}\}^L_{l=0}$ be a sequence of triangulations of $D$ such that $\mathcal{T}_{h_l}$ is obtained from $\mathcal{T}_{h_{l-1}}$ via uniform refinement with $h_{l}=\tfrac{1}{2}h_{l-1}=2^{-l}h_0$ for $l=1, \ldots, L$, where $h_l$ denotes the maximum mesh size of $\mathcal{T}_{h_l}$ and $h_0$ is the mesh size of an initial coarse triangulation $\mathcal{T}_{h_0}$. For any triangulation $\mathcal{T}_{h_l}$ we assume that
\[
\bar D= \bigcup_{T \in \mathcal{T}_{h_l}} \bar T, \quad h_l:=\max_{T \in \mathcal{T}_{h_l}}\mbox{diam}(T), \quad l=0, \ldots, L.
\]
Furthermore, on each $\mathcal{T}_{h_l}$  we construct the space of linear finite elements $X_{h_l}$ defined by  
\[
X_{h_l}:= \{ v \in C^0(\bar D) : v \mbox{ is a linear polynomial on each }  T \in \mathcal{T}_{h_l} \mbox{ and } {v}_{ | \partial D}=0   \}.
\]
Denoting by $N_l$ the inner nodes in the triangulation $\mathcal{T}_{h_l}$, we have $\dim(X_{h_l})=N_l$, where $\dim(X_{h_l})$ denotes the dimension of the space $X_{h_l}$. It is clear that for the spaces $X_{h_l}$ constructed this way there holds
\[
X_{h_0} \subset X_{h_1}  \subset \cdots \subset X_{h_L}.
\]

Finally, adopting the convention $h_l\approx N^{-\frac{1}{d}}_l$ with $d$ being the dimension of the computational domain $D$,  we make the following assumption on the computational cost of solving the problem $(\mathrm{P}_{\omega,h_l})$.
\begin{description}
\item[A4.]
For a given $\omega \in \Omega$ and mesh size $h_l$, the solution of $(\mathrm{P}_{\omega,h_l})$ satisfying the estimate \eqref{Eqn:29} can be computed with a computational cost $\mathcal{C}_l$ which is asymptotically, as $l \to \infty$, bounded by 
\begin{align*}
\mathcal{C}_l \lesssim h_l^{-\gamma} \approx N^{\frac{\gamma}{d}}_l
\end{align*}
with some real number $\gamma >0$, where $N_l:=\dim(X_{h_l})$ and $d$ is the dimension of the computational domain $D$.
\end{description}
It is worth to mention that the ideal value of $\gamma$ in the previous assumption would be $\gamma=d$, in this case for instance, doubling the number of unknowns $N_{l}$ should result in doubling the computational cost $\mathcal{C}_l$.

\subsection{The classical Monte Carlo method}
Let $u^\ast \in L^2(\Omega,L^2(D))$ be the random field defined by \eqref{Eqn:6}. The classical Monte Carlo estimator to $\mathbb{E}[u^\ast]$ is the sample average $E_M[u^\ast]$ defined by 
\begin{equation}
\label{Eqn:5}
E_M[u^\ast]:= \frac{1}{M}\sum_{i=1}^M u^\ast_{\omega_i},
\end{equation}
where $u^\ast_{\omega_i}:=u^\ast(\omega_i,\cdot), i=1,\ldots, M$ are $M$ independent identically distributed samples of $u^\ast$. Notice that, for a fixed $M$, the estimator \eqref{Eqn:5} can be interpreted as a $L^2(D)$-valued random variable. The next result gives the \emph{statistical error}  associated with the estimator \eqref{Eqn:5}.

\begin{theorem}
\label{Thm:5}
Let $u \in L^2(\Omega,L^2(D))$. Then, for any $M \in \mathbb{N}$, we have
\begin{equation}
\|\mathbb{E}[u]-E_M[u]\|_{L^2(\Omega,L^2(D))} \leq M^{-\frac{1}{2}} \|u\|_{L^2(\Omega,L^2(D))}. 
\end{equation}
\end{theorem}
\begin{proof}
Using  the fact that $u_{\omega_i}:=u(\omega_i, \cdot), i=1,\ldots,M $ are independent, identically distributed random samples, we obtain
\begin{align*}
\mathbb{E}\big[\|\mathbb{E}[u]-E_M[u]\|^2_{L^2(D)} \big] &=\mathbb{E}\Big[\Big\|\mathbb{E}[u]-\frac{1}{M}\sum_{i=1}^M  u_{\omega_i} \Big\|^2_{L^2(D)} \Big] \\
&=\frac{1}{M^2} \mathbb{E}\Big[\Big\|\sum_{i=1}^M \big(\mathbb{E}[u]- u_{\omega_i}\big) \Big\|^2_{L^2(D)} \Big] \\
&=\frac{1}{M^2} \sum_{i=1}^M \mathbb{E}\Big[\Big\| \mathbb{E}[u]- u_{\omega_i} \Big\|^2_{L^2(D)} \Big] \\
&=\frac{1}{M}  \mathbb{E}\big[\big\| \mathbb{E}[u]- u \big\|^2_{L^2(D)} \big] \\
&=\frac{1}{M} \Big( \mathbb{E}\big[\| u \|^2_{L^2(D)} \big] - \| \mathbb{E}[u] \|^2_{L^2(D)} \Big)\\
& \leq \frac{1}{M}  \mathbb{E}\big[\| u \|^2_{L^2(D)} \big]. 
\end{align*}
Taking the square root of both sides of the previous inequality gives the desired result. 
\end{proof}

It might be difficult in practice to obtain samples from  the random field $u^\ast$ since  it is not known most of the time. 
To overcome this problem we consider sampling from its finite element approximation $u^\ast_{h_L}$ at a given level $L\in \mathbb{N}$.
We use the classical finite element  Monte Carlo estimator to $\mathbb{E}[u^\ast]$ defined by
\begin{equation}
\label{Eqn:9}
E_M[u^\ast_{h_L}]:= \frac{1}{M}\sum_{i=1}^M u^\ast_{\omega_i,h_L},
\end{equation}
where $u^\ast_{\omega_i,h_L}:=u^\ast_{h_L}(\omega_i,\cdot), i=1,\ldots, M$ are $M$ independent identically distributed samples of $u^\ast_{h_L}$. The next theorem states the error estimate associated with \eqref{Eqn:9}.
 
\begin{theorem}
\label{Thm:6}
Let Assumptions A1-A3 hold for some $0<t \leq 1$. Then 
\begin{equation*}
\|\mathbb{E}[u^\ast]-E_M[ u^\ast_{h_L}]\|_{L^2(\Omega,L^2(D))} \leq C(\alpha,z,{U_{ad}}) ( M^{-\frac{1}{2}} + h^{2s}_L )
\end{equation*}
for all $0< s <t$ except $s=\tfrac{1}{2}$ where $C(\alpha,z,{U_{ad}}) >0$ is a constant independent of $h_L$ and depending only on the data $\alpha,z$ and on $U_{ad}$. If $t=1$, the above  estimate hold with $s=1$.
\end{theorem} 
\begin{proof}
We start the proof by using the triangle inequality to obtain
\begin{equation*}
\|\mathbb{E}[ u^\ast]-E_M[ u^\ast_{h_L}]\|_{L^2(\Omega,L^2(D))} \leq \|\mathbb{E}[u^\ast]-\mathbb{E} [u^\ast_{h_L}]\|_{L^2(\Omega,L^2(D))} + \|\mathbb{E}[ u^\ast_{h_L}]-E_M[u^\ast_{h_L}]\|_{L^2(\Omega,L^2(D))}.
\end{equation*}
The task  is now  to estimate the two terms on the right hand side of the previous inequality. The  estimate for the first term follows from Theorem~\ref{Thm:3} after utilizing the Cauchy-Schwarz inequality. In fact, we have
\begin{align*}
\|\mathbb{E}[ u^\ast]-\mathbb{E} [u^\ast_{h_L}]\|_{L^2(\Omega,L^2(D))} &=\|\mathbb{E}[u^\ast-u^\ast_{h_L}]\|_{L^2(\Omega,L^2(D))} =\|\mathbb{E}[u^\ast-u^\ast_{h_L}]\|_{L^2(D)} \\
&\leq \mathbb{E} [ \|u^\ast-u^\ast_{h_L}\|_{L^2(D)} ] \leq \| u^\ast-u^\ast_{h_L}\|_{L^2(\Omega,L^2(D))} \\
&\leq C(\alpha,z,{U_{ad}})  h^{2s}_L.
\end{align*}
For the second term, we use Theorem~\ref{Thm:5} together with the bound \eqref{Eqn:33} to obtain 
\begin{align*}
\|\mathbb{E}[u^\ast_{h_L}]-E_M[u^\ast_{h_L}]\|_{L^2(\Omega,L^2(D))} &\leq M^{-\frac{1}{2}} \|u^\ast_{h_L}\|_{L^2(\Omega,L^2(D))} \\
& \leq C(\alpha,z,{U_{ad}}) M^{-\frac{1}{2}}.
\end{align*}
Combining the estimates of the two terms gives the desired result. 
\end{proof}

The previous theorem tells us that the total error resulting from using \eqref{Eqn:9} as an approximation to $\mathbb{E}[u^\ast]$ can be decomposed into two parts; a statistical part which is of order $M^{-1/2}$ and a discretization part of order $h^{2s}_L$. This suggests that the number of samples $M$ should be related to the mesh size $h_L$ in order to achieve a certain overall error. We state this more rigorously in the next theorem and we give the total computational cost of using \eqref{Eqn:9}.

\begin{theorem}
Let Assumptions A1-A4 hold for some $0<t \leq 1$. Then, the MC estimator \eqref{Eqn:9} with the following choice of number of samples
\begin{align*}
M = O({h}^{-4s}_L), 
\end{align*}
yields the error bound 
\begin{align} 
\label{Eqn:30}
\|\mathbb{E}[ u^\ast]-E_M[ u^\ast_{h_L}]\|_{L^2(\Omega,L^2(D))} \leq C(\alpha,z,{U_{ad}}) h^{2s}_L 
\end{align}
for all $0< s <t$ except $s=\tfrac{1}{2}$, with a total computational cost $\mathcal{C}_L$ which is asymptotically, as $L \to \infty$, bounded by 
\begin{align}
\label{Eqn:31}
\mathcal{C}_L \lesssim h^{-\gamma-4s}_L,
\end{align}
for some $C(\alpha,z,{U_{ad}})>0$ depending on the data $\alpha,z$ and on $U_{ad}$. If $t=1$, the above  estimates \eqref{Eqn:30} and \eqref{Eqn:31} hold with $s=1$. 
\end{theorem}
\begin{proof}
The estimate \eqref{Eqn:30} follows from Theorem~\ref{Thm:6} after choosing $M  \approx {h}^{-4s}_L$. To obtain the bound \eqref{Eqn:31}, it is sufficient to multiply the computational cost of one sample from Assumption A4 by the total number of samples $M \approx {h}^{-4s}_L$.
\end{proof}

\subsection{The multilevel Monte Carlo method}\label{sec:MLMC}

We start by observing that the random field $u^\ast_{h_L}$ can be written as
\begin{equation*}
u^\ast_{h_L}=  \sum^L_{l=0} (u^\ast_{h_l}-u^\ast_{h_{l-1}}),
\end{equation*}
where $u^\ast_{h_{-1}}:=0$. The linearity of the expectation operator $\mathbb{E}[\cdot]$ implies
\begin{equation}
\label{Eqn:21}
\mathbb{E}[u^\ast_{h_L}]=  \sum^L_{l=0} \mathbb{E}[u^\ast_{h_l}-u^\ast_{h_{l-1}}].
\end{equation}
In the multilevel Monte Carlo method, we approximate $\mathbb{E}[u^\ast_{h_l}-u^\ast_{h_{l-1}}]$ in \eqref{Eqn:21} by the classical Monte Carlo estimator \eqref{Eqn:9} with a number of samples $M_l$ that depends on the mesh level $l$. Therefore, the MLMC estimator to $\mathbb{E}[u^\ast]$ reads
\begin{equation}
\label{Eqn:10}
E^L[u^\ast_{h_L}]:=  \sum^L_{l=0} E_{M_l}[u^\ast_{h_l}-u^\ast_{h_{l-1}}],
\end{equation}
where the samples over all levels are independent of each other. The next theorem gives  the error estimate associated with the estimator \eqref{Eqn:10}. 
\begin{theorem}
\label{Thm:13} 
Let Assumptions A1-A3 hold for some $0<t \leq 1$. Then
\begin{align}
\label{Eqn:23} 
\|\mathbb{E}[ u^\ast]-E^L[ u^\ast_{h_L}]\|_{L^2(\Omega,L^2(D))} \leq C(\alpha,z,{U_{ad}})  \Big( h^{2s}_L +  \sum^L_{l=0}  M^{-\frac{1}{2}}_l h^{2s}_l \Big),
\end{align}
for all $0< s <t$ except $s=\tfrac{1}{2}$, where $C(\alpha,z,{U_{ad}})>0$ depends on the data $\alpha,z$ and on $U_{ad}$. If $t=1$, the above  estimate holds with $s=1$.
\end{theorem}

\begin{proof}
Throughout the proof, we use the notation $\|\cdot\|_V:=\|\cdot\|_{L^2(\Omega,L^2(D))}$. We start by observing that using the triangle inequality together with \eqref{Eqn:21}, \eqref{Eqn:10} gives
\begin{align}
\label{Eqn:22}
\|\mathbb{E}[ u^\ast]-E^L[ u^\ast_{h_L}]\|_V &\leq \|\mathbb{E}[ u^\ast]-\mathbb{E} [u^\ast_{h_L}]\|_V + \|\mathbb{E}[ u^\ast_{h_L}]-E^L[ u^\ast_{h_L}]\|_V \nonumber \\
&\leq I + II, 
\end{align}
where we define
\begin{equation*}
I:= \|\mathbb{E}[ u^\ast]-\mathbb{E} [u^\ast_{h_L}]\|_V \quad \mbox{ and } \quad II:=\sum^L_{l=0} \|\mathbb{E}[ u^\ast_{h_l}-u^\ast_{h_{l-1}}]-E_{M_l}[ u^\ast_{h_l}-u^\ast_{h_{l-1}}]\|_V.
\end{equation*}
The aim is now to estimate the terms $I$ and $II$. We start by estimating $I$. To this end,  it is enough to argue like in the proof of Theorem~\ref{Thm:6} to obtain
\begin{align*}
\|\mathbb{E}[ u^\ast]-\mathbb{E} [u^\ast_{h_L}]\|_V \leq C(\alpha,z,{U_{ad}}) h^{2s}_L.
\end{align*}
To estimate the term $II$, we utilize Theorem~\ref{Thm:5}, the triangle inequality, Theorem~\ref{Thm:3} and the fact that $h_{l-1}=2 h_l$ to get
\begin{align*}
\sum^L_{l=0} \|\mathbb{E}[ u^\ast_{h_l}-u^\ast_{h_{l-1}}]-E_{M_l}[ u^\ast_{h_l}-u^\ast_{h_{l-1}}]\|_V  &\leq  \sum^L_{l=0}  M^{-\frac{1}{2}}_l  \|u^\ast_{h_l}-u^\ast_{h_{l-1}}\|_V  \\
& \leq  \sum^L_{l=0}  M^{-\frac{1}{2}}_l  \Big(\| u^\ast_{h_l}-u^\ast\|_V + \|u^\ast-u^\ast_{h_{l-1}}\|_V \Big) \\
& \leq C(\alpha,z,{U_{ad}}) \sum^L_{l=0}  M^{-\frac{1}{2}}_l ( h^{2s}_l +  h^{2s}_{l-1}) \\
& = C(\alpha,z,{U_{ad}})\sum^L_{l=0}  M^{-\frac{1}{2}}_l ( 1 +  2^{2s})h^{2s}_l.
\end{align*}
Hence, combining the estimates of the terms $I$, $II$  and inserting them in \eqref{Eqn:22} gives
\begin{align*}
\|\mathbb{E}[ u^\ast]-E^L[ u^\ast_{h_L}]\|_V \leq  C(\alpha,z,{U_{ad}}) \Big( h^{2s}_L + \sum^L_{l=0}  M^{-\frac{1}{2}}_l h^{2s}_l \Big),
\end{align*}
which is the desired result and the proof is complete. 
\end{proof}

The previous theorem holds for any choice of $\{M_l\}^L_{l=0}$ in \eqref{Eqn:10}, where $M_l$ is the  number of samples over the refinement level $l$.  However, it is desirable that  $\{M_l\}^L_{l=0}$ is chosen in such a way that the statistical error and the discretization error in \eqref{Eqn:23} are balanced.  The next theorem suggests a choice of $\{M_l\}^L_{l=0}$ such that the  overall error in  \eqref{Eqn:23} is of order $h^{2s}_L$ and it gives the associated computational cost.

\begin{theorem}
\label{Thm:14} 
Let Assumptions A1-A4 hold for some $0<t \leq 1$. Then, the  MLMC estimator \eqref{Eqn:10} with the following choice of $\{M_l\}^L_{l=0}$ where
\begin{align*}
M_l = 
\begin{cases}
O(h_L^{-4s} h_l^{\frac{\gamma+4s}{2}}),   & 4s>\gamma \\
O\big( (L+1)^2 h_L^{-4s} h_l^{\frac{\gamma+4s}{2}} \big),   & 4s=\gamma \\
O( h_L^{-\frac{\gamma+4s}{2}} h_l^{\frac{\gamma+4s}{2}}),   & 4s<\gamma
\end{cases}
\end{align*}
yields the error bound
\begin{align}
\label{Eqn:24} 
\|\mathbb{E}[ u^\ast]-E^L[ u^\ast_{h_L}]\|_{L^2(\Omega,L^2(D))} \leq C(\alpha,z,{U_{ad}}) h^{2s}_L 
\end{align}
for all $0< s <t$ except $s=\tfrac{1}{2}$, with a total computational cost $\mathcal{C}_L$ which is asymptotically, as $L \to \infty$, bounded by 
\begin{align}
\label{Eqn:25}
\mathcal{C}_L \lesssim 
\begin{cases}
h_L^{-4s},   & 4s>\gamma \\
(L+1)^3 h_L^{-4s},   & 4s=\gamma \\
h_L^{-\gamma},   & 4s<\gamma.
\end{cases}
\end{align}
Here,  $C(\alpha,z,{U_{ad}}) >0$ depends on the data $\alpha,z$ and on $U_{ad}$. If $t=1$, the above  estimates \eqref{Eqn:24} and \eqref{Eqn:25} hold with $s=1$. 
\end{theorem}

\begin{proof}
We give the proof only for the case $4s>\gamma$; the other two cases $4s=\gamma$ and $4s<\gamma$ can be treated analogously. To verify the estimate \eqref{Eqn:24} it is enough to utilize  Theorem~\ref{Thm:13} together with the choice
\begin{align}
\label{Eqn:26}
M_l  \approx h_L^{-4s} h_l^{\frac{\gamma+4s}{2}}, \quad l=0,\ldots,L,
\end{align}
and the approximation  $h_l \approx 2^{-l}$ to obtain 
\begin{align*}
\|\mathbb{E}[ u^\ast]-E^L[ u^\ast_{h_L}]\|_{L^2(\Omega,L^2(D))} &\leq C(\alpha,z,{U_{ad}}) \Big( h^{2s}_L +  \sum^L_{l=0}  M^{-\frac{1}{2}}_l h^{2s}_l \Big)\\
& = C(\alpha,z,{U_{ad}}) \Big( h^{2s}_L + h^{2s}_L \sum^L_{l=0}  h_l^{\frac{4s-\gamma}{4}} \Big)\\
& = C(\alpha,z,{U_{ad}}) \Big( h^{2s}_L + h^{2s}_L \sum^L_{l=0}  2^{-(\frac{4s-\gamma}{4}) l} \Big)\\
& = C(\alpha,z,{U_{ad}}) h^{2s}_L   \Bigg( 1 +  \dfrac{2^{-(\frac{4s-\gamma}{4})(L+1)}-1}{2^{-\frac{4s-\gamma}{4}}-1}  \Bigg) \\
&  \leq C(\alpha,z,{U_{ad}}) h^{2s}_L.
\end{align*}

It remains to verify the asymptotic upper bound for the total computational cost \eqref{Eqn:25}. To achieve this, we see that from Assumption A4 together with the choice \eqref{Eqn:26} and $h_l \approx 2^{-l}$, we have
\begin{align*}
\mathcal{C}_L=\sum^L_{l=0} M_l \mathcal{C}_l & \lesssim \sum^L_{l=0} h_L^{-4s} h_l^{\frac{\gamma+4s}{2}} h_l^{-\gamma} = h_L^{-4s} \sum^L_{l=0}  h_l^{\frac{4s-\gamma}{2}} \approx h_L^{-4s} \sum^L_{l=0}  2^{-(\frac{4s-\gamma}{2})l}  \\
& \approx  h_L^{-4s} \frac{2^{-(\frac{4s-\gamma}{2})(L+1)}-1}{2^{-(\frac{4s-\gamma}{2})}-1} \lesssim h_L^{-4s}, \qquad \mbox{ as }L \to \infty,
\end{align*}   
which is the desired result.
\end{proof}

Importantly, by comparing the total cost of MC in \eqref{Eqn:31} and MLMC in \eqref{Eqn:25} we see that the  multilevel estimator achieves the same accuracy as classical Monte Carlo at a fraction of the cost.

We remark that the hidden constant in $O(\cdot)$ in the sequence $\{M_l\}^L_{l=0}$ from Theorem~\ref{Thm:14} plays a crucial rule in determining the size of the statistical error. 
This can be seen in \eqref{Eqn:23} where it is clear that the larger the value of the constant in $O(\cdot)$, the smaller the statistical error. 
In order to obtain a minimal choice of $\{M_l\}^L_{l=0}$ we adapt the strategy presented in \cite[Remark~4.11]{KSW:2014}, that is, $\{M_l\}^L_{l=0}$ is chosen to be the solution of the following minimization problem
\begin{equation}
\label{Eqn:36}
(\mathrm{PN}) \quad
\begin{array}{l}
 \min_{\bar M \in \mathcal{M}_{ad}} \mathcal{J}(\bar M):=\sum^L_{l=0} M_l \mathcal{C}_l \\
\end{array}
\end{equation} 
where 
\begin{align*}
\mathcal{M}_{ad}:=\big\{(M_1,\ldots,M_L) \in \mathbb{N}^L: M_l \geq 1 \mbox{ for }l=0,\ldots,L \mbox{ and } \\
 \sum^L_{l=0}  M^{-\frac{1}{2}}_l h^{2s}_l \leq c_0 h^{2s}_L,  \mbox{ for fixed } c_0>0 \big\}.
\end{align*}
The problem $(\mathrm{PN})$ is a convex minimization problem. Moreover, for a fixed $c_0>0$, the set $\mathcal{M}_{ad}$ is non-empty since $\{M_l\}^L_{l=0}$ from Theorem~\ref{Thm:14} belongs to $\mathcal{M}_{ad}$ if the hidden constant in $O(\cdot)$ is large enough. It should be clear that the solution   $\{M^\ast_l\}^L_{l=0}$ of $(\mathrm{PN})$ is also satisfying \eqref{Eqn:26} since $\mathcal{J}(\{M^\ast_l\}^L_{l=0}) \leq \mathcal{J}(\{M_l\}^L_{l=0})=:\mathcal{C}_L$. 

\begin{remark}
Observe that the admissible set of controls $U_{ad}$ is a \textit{convex} set.
However, it is clear that the MLMC estimate for $\mathbb{E}[u^\ast]$ in \eqref{Eqn:10} is in general not admissible since the corrections  in \eqref{Eqn:10} are computed using different realizations of the random coefficient.
In contrast, the classical MC estimate in \eqref{Eqn:5} is always admissible since it is a convex combination of admissible controls.
This has already been observed in the context of random obstacle problems \cite{BierigChernov:2014,KSW:2014}.
\end{remark}

\section{Numerical Example}\label{sec:example}
In this section we verify numerically the assertion of Theorem~\ref{Thm:14}, namely, the order of convergence \eqref{Eqn:24} and the upper bound \eqref{Eqn:25} for the computational cost. 
For this purpose, we consider the optimal control problem
\begin{align}
\label{Eqn:34} 
\min_{u \in U_{ad}} J_\omega(u)=\frac{1}{2} \Vert y_\omega-z \Vert_{L^2(D)}^2  + \dfrac{\alpha}{2} \Vert u \Vert_{L^2(D)} ^2 
\end{align}
subject to
\begin{equation*}
 \begin{aligned}
-\nabla \cdot ( a(\omega,x) \nabla y(\omega,x) )  &=u(x)  \quad \mbox{ in }  D, \\
y(\omega, x) &=0  \quad \mbox{ on } \partial D,
\end{aligned}
\end{equation*}
where we define $D:=(-0.5,0.5)\times(-0.5,0.5) \subset \mathbb{R}^2$ and $ U_{ad}:=L^2(D)$.
The data is chosen as follows: 
\begin{align} 
\alpha &=10^{-2}, \nonumber\\
z(x) &= \sin(2 \pi x_1) \cos(\pi x_2),\nonumber\\
a(\omega,x)&=e^{\kappa(\omega,x)}, \label{Eqn:35} 
\end{align}
with the random field $\kappa$ defined by
\begin{align*}
{\kappa}(x,\omega):=&0.84\cos(0.42\pi x_1)\cos(0.42\pi x_2) Y_1(\omega) +0.45\cos(0.42\pi x_1) \sin(1.17\pi x_2) Y_2(\omega)\\  
+& 0.45\sin(1.17\pi x_1)\cos(0.42\pi x_2)Y_3(\omega)+0.25 \sin(1.17\pi x_1)\sin(1.17\pi x_2) Y_4(\omega),
\end{align*}
where $Y_i \sim N(0,1),i=1,\ldots,4,$ are independent normally distributed random variables. 
In fact, the random field $\kappa$  approximates a \textit{Gaussian} random field with zero mean and covariance function $C(x,\widetilde x)=e^{-\|x-\widetilde x\|_1}, x,\widetilde x \in D$, where $\|\cdot\|_1$ denotes the $l_1$-norm in $\mathbb{R}^2$. 
The terms in $\color{blue}\kappa$ are the four leading terms in the associated Karhunen-Lo\`eve expansion,
see \cite{ghanem2003stochastic} for more details.
As a consequence, the random field $a$ in \eqref{Eqn:35} is a (truncated) lognormal field.

Assumptions A1--A2 are satisfied for all $t<1/2$ for any lognormal random field $a$ where $\log(a)$ has a Lipschitz continuous, isotropic covariance function and a mean function in $C^t(\overline{D})$, see \cite[Proposition~2.4]{charrier2013finite}.
The property $1/a_{\min} \in L^p(\Omega)$ for all $p\in [1,\infty)$ is proved in \cite[Proposition~2.3]{Charrier:2012}.
In our example the covariance function of $\kappa$ is in fact analytic in $\overline{D}\times\overline{D}$.
This gives realizations of $\kappa=\log(a)$ (and thus $a$) which belong to $C^1(\overline D)$ almost surely.
Hence Assumption A2 is satisfied for $t= 1$.

For a given realization of the coefficient $a(\omega,x)$, the problem in \eqref{Eqn:34} is discretized by means of the variational discretization as described in \S\ref{sec:FE}.
The resulting discrete optimality system is solved numerically by a semi-smooth Newton method, see for instance \cite{hinze2012globalized}. 
All the computations are done using  a Matlab implementation running on  3.40 GHz 4$\times$Intel Core i5-3570 processor with 7.8 GByte of RAM. 
For solving the linear system in each iteration of the semi-smooth Newton method, we use the Matlab backslash operation.
Note that the linear system is a $3 \times 3$ block system of order $O(N_l)$ with sparse blocks.
Hence the backslash costs about $O(N_l^{1.5}) = O(h_l^{-1.5 d})$ operations in $d$-dimensional space.
A numerical study in \cite{hinze2012globalized} reveals that the number $n$ of semi-smooth Newton iterations is independent of the mesh size $h_l$; in fact it depends on the input data, the tolerance, and the initial guess.
In our context this means that $n$ depends on the realization $a_\omega$ of the diffusion coefficient, on the parameter $\alpha$ and on the desired state $z$ in the cost functional in \eqref{Eqn:34}.
Since $\alpha$ and $z$ are in general supplied by the user we do not consider variations of these parameters here.
For our numerical example we found that $n$ does not vary significantly across the realizations $a_\omega$, see \S\ref{sec:cost}. 
In summary, assuming that the number of semi-smooth Newton iterations is independent of the realizations $a_\omega$ and the level $l$, the cost to obtain one sample of the optimal control is $O(h_l^{-1.5d})$.
Hence Assumption A4 is satisfied with $\gamma=1.5\times d$.
We mention that it is possible to achieve the ideal value $\gamma=d$ by using a multigrid based method (see e.g. \cite{GongXieYan:2015}).

\subsection{FE convergence rate and computational cost}\label{sec:cost}
Observe that Theorem~\ref{Thm:14} requires the values of $\gamma$ and $s$ a priori.
These can be estimated easily via numerical computations as illustrated in Figure~\ref{Fig:1}. 
The value of $\gamma$ for our solver can be deduced from Figure~\ref{Fig:1a} where we plot of the average cost (CPU-time in seconds) for computing $u_{h_l}$, the solution of \eqref{Eqn:34}, for a given realization $a(\omega,x)$ and mesh size $h_l$ versus the number of degrees of freedom $N_l$ in the mesh when  $h_l=2^{-l}$  for $l=0,\ldots,8$. 
We see in the figure that the asymptotic behavior of the average cost is $O(N_l^{1.2})$ and thus $\gamma \approx 2.4$.
This is slightly better than $\gamma=1.5\times d=3$ which we expect in 2D space.
Here, the average cost at a given $N_l$ is considered to be the average of the total CPU-time in seconds required to solve \eqref{Eqn:34} for 500 independent realizations of the coefficient $a(\omega,x)$ at the given mesh size $h_l$. 
To confirm that the cost per sample does not vary significantly across the realizations $a_\omega$ we plot the CPU-time in seconds with respect to $N_l$ for individual realizations of $a_\omega$ in Figure~\ref{Fig:1c}.

The value of $s$ can be obtained from Figure~\ref{Fig:1b} where we plot $ E_{500} [ \|u_{\omega,h^\ast}-u_{\omega,h_l}\|_{L^2(D)} ]$ versus $h_l^{-1}=2^l$ for $l=0,\ldots,7$. 
Here, $E_{500} [\cdot]$ denotes the sample average of $500$ independent samples. 
Furthermore, $u_{\omega,h_l}$ is the solution of \eqref{Eqn:34} at a given mesh size $h_h$ and realization of $a(\omega,x)$.
The control $u_{\omega,h^\ast}$ with $h^\ast:=2^{-8}$ is a reference solution since the exact solution of \eqref{Eqn:34} is not known. 
We see clearly from the plot that the asymptotic behavior of the error is  $O(h_l^{2})$ as $h_l \to 0$, and thus $s=1$. 
In fact, this quadratic order of convergence should be expected since the realizations  of \eqref{Eqn:35} belong to $C^t(\overline D)$ with $t=1$ and according to Theorem~\ref{Thm:3} we have $s=1$ if $t=1$. 
Furthermore,  we observe that the  error enters the asymptotic regime  when the mesh size is $h_2=2^{-2}$ or smaller. 
This suggests that in the MLMC estimator \eqref{Eqn:21} 
one should choose the mesh size $h_0$ of the coarsest level to be $h_0=2^{-2}$.
 Finally, for  all the experiments  used in Figure~\ref{Fig:1}, the triangulation of the domain $D$ for $l=0$ consists of four triangles with only one degree of freedom located at the origin. 

\subsection{Multilevel Monte Carlo simulation}
Having estimated the values of $\gamma$ and $s$, we are in a position to verify the error estimate \eqref{Eqn:24} and the upper bound for the computational cost in \eqref{Eqn:25} for the MLMC estimator $E^L[u^\ast_{h_L}]$, where $u^\ast_{h_L}$ is the random field associated to \eqref{Eqn:34} as defined in \eqref{Eqn:7}. 
To this end, let $\{\mathcal{T}_{h_l}\}^L_{l=0}$, for $L=1,\ldots,5$, be sequences of triangulations of the domain $D$ as described at the beginning of \S\ref{sec:MC}. 
Here, we choose the mesh size $h_0$ of the initial coarse triangulation  $\mathcal{T}_{h_0}$ to be $h_0=2^{-2}$ (see the previous paragraph for the reason of this choice). 
Since the expected value $\mathbb{E}[u^\ast]$ is not known explicitly, we consider the MLMC estimator $E^{L^\ast}[u^\ast_{h_{L^\ast}}]$ to be the reference expected value where $L^\ast=6$ and $h_{L^\ast}=2^{-8}$. 

It is clear that the asymptotic behavior of the error $\mathbb{E}[u^\ast]-E^L[u^\ast_{h_L}]$ in the $L^2(\Omega,L^2(D))$-norm and in the $L^2(D)$-norm is the same. 
To simplify the computations we thus calculate the error in the $L^2(D)$-norm.  
Finally, for any given value of $L$, we obtain the sequence $\{M_l\}^L_{l=0}$ of number of samples per refinement level through solving \eqref{Eqn:36} with the choice $c_0=\tfrac{1}{2}$ by the \texttt{fmincon} function from the  Matlab Optimization Toolbox.
We round non-integer  values in the sequence  using the ceiling function. 
In Table~\ref{table:1}, we report the sequences $\{M_l\}^L_{l=0}$, for $L=1,\ldots,5$, used in computing $E^L[u^\ast_{h_L}]$, where $M_l$ is the number of samples for a refinement level with mesh size $h_l=2^{-(2+l)}$. 

Figure~\ref{Fig:2a} presents the plot of the CPU-time (in seconds) for the computation of $E^L[u^\ast_{h_L}]$ vs. the number of degrees of freedom $N_L$ in triangulations with mesh size $h_L=2^{-(2+L)}$, for $L=1,\ldots,5$. 
It is clear from the figure that the computational cost is asymptotically bounded by $O(N^2_L)$ as $L \to \infty$.
Since $N_L = O(h_L^{-2})$, this confirms the theoretical cost bound in \eqref{Eqn:25} in the case $4s > \gamma$ (recall that $s=1$ and $\gamma \approx 2.4$ in problem \eqref{Eqn:34}).
In fact, the theoretical cost bound is sharp in this case.

Note that we did not verify the cost bound for the MC estimator in \eqref{Eqn:31} due to limited computational time.
In our example the MC estimator requires $O(h_L^{-2.4})$ more operations on level $L$ than the MLMC estimator to achieve the same accuracy.

In addition we report the error associated with $E^L[u^\ast_{h_L}]$ in Figure~\ref{Fig:2b}.  
We can see clearly that the best fitting curve for the error behaves like $O(h^2_L)$ as $L \to \infty$.
This is predicted by \eqref{Eqn:24}.


\begin{table}[ht]
\centering
\caption{The sequences $\{M_l\}^L_{l=0}$, for $L=1,\ldots,5$, used in computing $E^L[u^\ast_{h_L}]$, where $M_l$ is the number of samples over a refinement level with mesh size $h_l=2^{-(2+l)}$.}
\label{table:1}
\begin{tabular}{ l | l  l  l l l l}
\hline
$L$ &	  $M_0$&       $M_1$&   $M_2$&  $M_3$& $M_4$ & $M_5$  \\
\hline

1 &	  183      &     24      &	 	        &            &             &        \\
2 &	  4815     &     631     &      83     &            &             &        \\
3 &	  102258   &     13387   &      1753   &      230   &             &        \\
4 &	  1948277  &     255053  &      33390  &      4372  &       573   &        \\
5 &	  34878076 &     4565950 &      597737 &      78251 &       10244 &   1342 \\

\hline 
\end{tabular}
\end{table}

\begin{figure}[p]
        \centering
        \begin{subfigure}[h!]{0.8\textwidth}
                \includegraphics[trim = 33mm 80mm 40mm 70mm, clip, width=\textwidth]{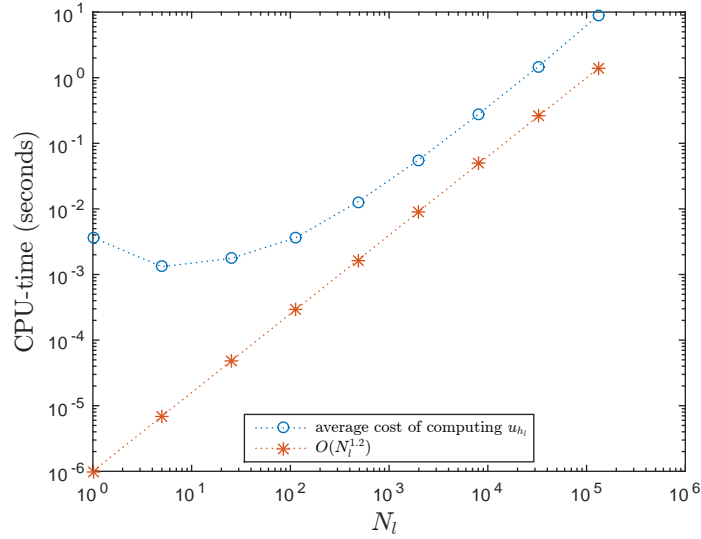}
                \caption{The average cost (CPU-time in seconds) of computing $u_{h_l}$ the solution of \eqref{Eqn:34} for a given realization of $a(\omega,x)$ vs. the number of degrees of freedom $N_l$ when $h_l=2^{-l}$ for $l=0,\ldots,8$. } 
         \label{Fig:1a}       
        \end{subfigure}%
        
        \begin{subfigure}[h!]{0.8\textwidth}
                \includegraphics[trim = 38mm 80mm 40mm 70mm, clip, width=\textwidth]{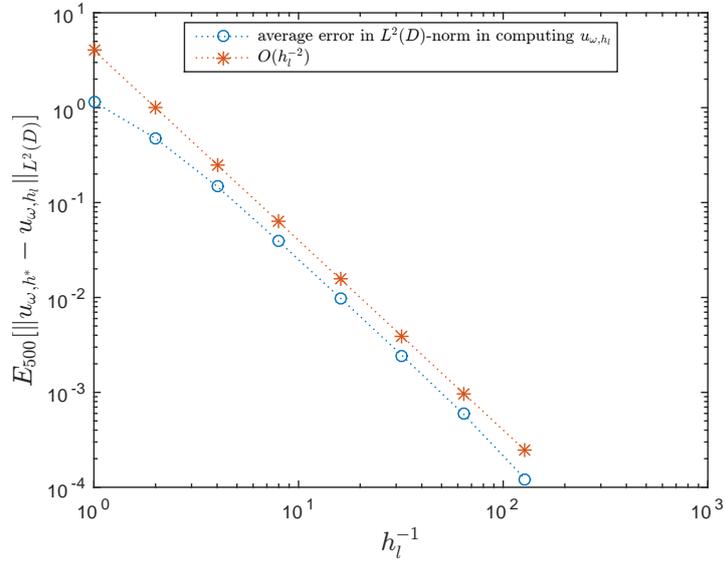}
                \caption{The average error $E_{500}[ \|u_{\omega,h^\ast}-u_{\omega,h_l}\|_{L^2(D)} ]$ vs. $h_l^{-1}=2^l$ for $l=0,\ldots,7$, where $u_{\omega,h_l}$ is the solution of \eqref{Eqn:34} for a given realization $a(\omega,x)$ and $u_{\omega,h^\ast}$ the reference solution with $h^\ast=2^{-8}$.}
         \label{Fig:1b}
        \end{subfigure}

        \caption{The computations of $s$ and $\gamma$ for  the estimate \eqref{Eqn:29} and  Assumption~A4, respectively.}
        \label{Fig:1}
\end{figure}

\begin{figure}[p]
        \centering
        \begin{subfigure}[h!]{0.8\textwidth}
                \includegraphics[trim = 33mm 80mm 40mm 70mm, clip, width=\textwidth]{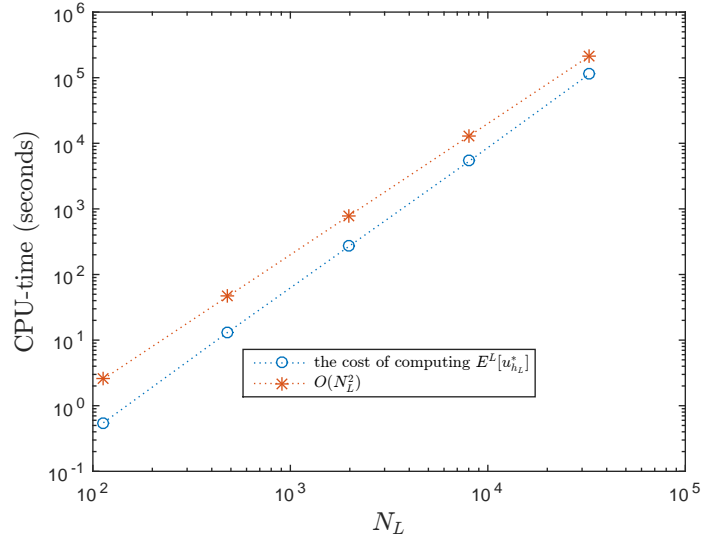}
                \caption{The cost (CPU-time in seconds) of computing $E^L[u^\ast_{h_L}]$ vs. the number of degrees of freedom $N_L$ when $h_L=2^{-(L+2)}$ for $L=1,\ldots,5$. } 
         \label{Fig:2a}       
        \end{subfigure}%
        
        \begin{subfigure}[h!]{0.8\textwidth}
                \includegraphics[trim = 36mm 80mm 40mm 70mm, clip, width=\textwidth]{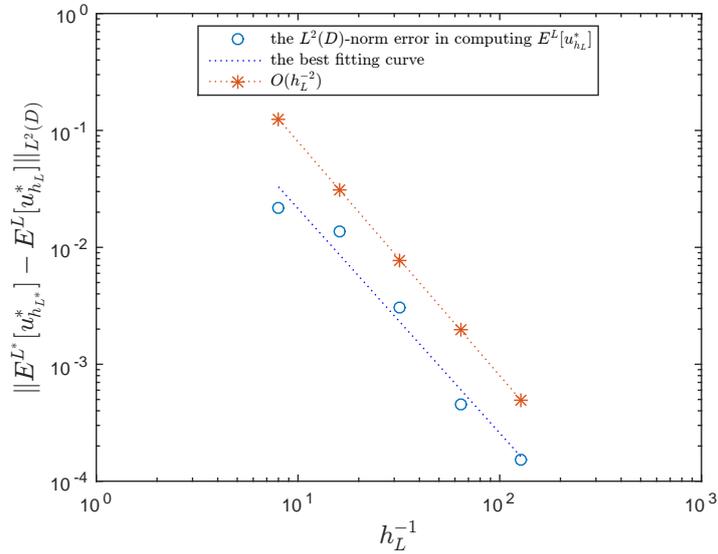}
                \caption{The error $\|E^{L^\ast}[u^\ast_{h_{L^\ast}}]-E^{L}[u^\ast_{h_{L}}]\|_{L^2(D)} $ vs. $h_L^{-1}=2^{(L+2)}$ for $L=1,\ldots,5$, where  $E^{L^\ast}[u^\ast_{h_{L^\ast}}]$ with $L^\ast=6$ is the reference expected value.}
         \label{Fig:2b}
        \end{subfigure}

        \caption{The error order of convergence and the computational cost upper bound for the MLMC estimator $E^{L}[u^\ast_{h_L}]$, where $u^\ast_{h_L}$ is the random field associated to \eqref{Eqn:34} as defined in \eqref{Eqn:7}.}
        \label{Fig:2}
\end{figure}

\begin{figure}[p]
        \centering
                \includegraphics[trim = 33mm 80mm 40mm 70mm, clip, width=\textwidth]{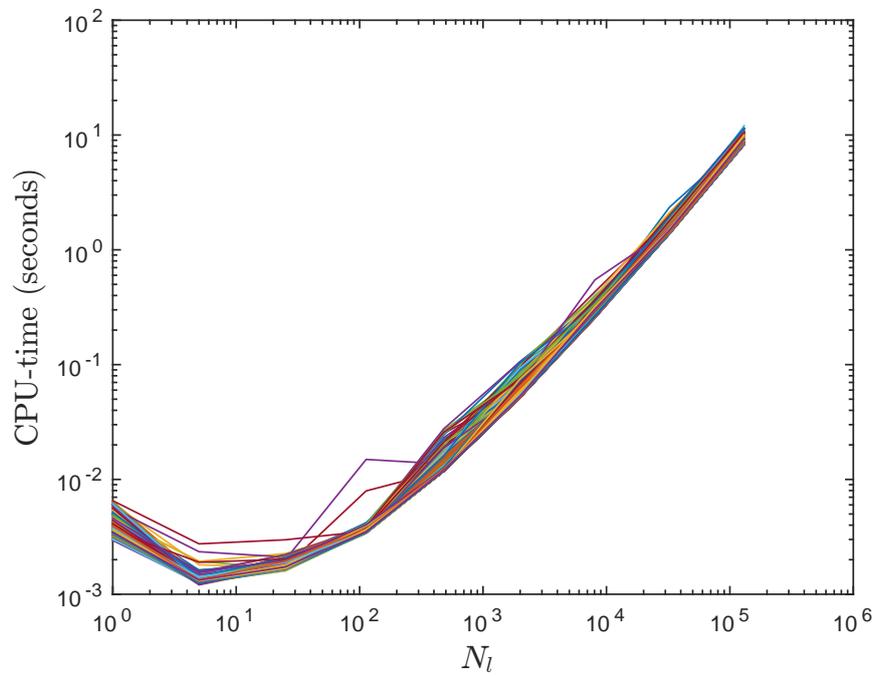}
                \caption{The cost (CPU-time in seconds) of computing $u_{h_l}$ the solution of \eqref{Eqn:34} for a given realization of $a(\omega,x)$ vs. the number of degrees of freedom $N_l$ when $h_l=2^{-l}$ with $l=0,\ldots,8$ for 500 independent realizations.} 
         \label{Fig:1c}      
\end{figure}

%
%

\bibliographystyle{plain}

\end{document}